\documentclass[12pt]{article}
\usepackage{amsmath,amssymb,amsthm}
\numberwithin{equation}{section}
\usepackage{hyperref}
\usepackage{units}
\usepackage{color}
\usepackage[T1]{fontenc}
\usepackage[utf8]{inputenc}
\usepackage{authblk}
\usepackage{bm}
\usepackage{graphicx}

\usepackage{datetime}

\usepackage[toc,page]{appendix}

\newtheorem{thm}{Theorem}[section]

\newtheorem{lemma}{Lemma}

\newtheorem{cor}[thm]{Corollary}

\title{Evaluation of weighted Fibonacci sums of a certain type}

\author[]{Kunle Adegoke \thanks{adegoke00@gmail.com}}

\affil{Department of Physics and Engineering Physics, \mbox{Obafemi Awolowo University}, 220005 Ile-Ife, Nigeria}

\begin{document}

\date{}

\maketitle

\begin{abstract}
\noindent  We derive a formula for the evaluation of weighted generalized Fibonacci sums of the type $S_k^n (w,r) = \sum_{j = 0}^k {w^j j^r G_j{}^n }$. Several explicit evaluations are presented as examples. 
\end{abstract}

\section{Introduction}
The generalized Fibonacci numbers $G_i$, $i \ge 0$, are defined through the second order recurrence relation $G_{i+1} = G_i + G_{i-1}$, with arbitrary seeds $G_0$ and $G_1$. In particular, when $G_0 = 0$ and $G_1= 1$, we have the Fibonacci numbers, denoted $F_i$, and when $G_0= 2$ and $G_1= 1$, we have the Lucas numbers, $L_i$. Extension to negative index is provided through \mbox{$G_{-i}=(-1)^i(F_{i+1}G_0-F_iG_1)$}, Vajda~\cite{vajda}, Formula~(9).

\bigskip

In this paper we will derive a formula (Theorem~\ref{thm.main}) for the evaluation of weighted generalized Fibonacci sums of the type
\begin{equation}\label{eq.z5p75mn}
S_k^n (w,r) = \sum_{j = 0}^k {w^j j^r G_j{}^n }\,,
\end{equation}
where $w$ is a real or complex parameter and $r$, $k$ and $n$ are nonnegative integers. For brevity, $S_k^1 (w,r)$ will be denoted $S_k(w,r)$ and $S_k^1 (1,r)$ will often be denoted simply by $S_k(r)$.  Similarly, we will write $\overline S_k(r)$ for $S_k^1 (-1,r)$. The evaluation of \mbox{$S_k (r)=\sum_{j=0}^k{j^rG_j}$} has received some attention in the literature, see for example references~\cite{koshy01,koshy,clarke,gauthier95,ozeki05}.

\bigskip

Since $S_k^n(w,0)$ in~\eqref{eq.z5p75mn} is to give $ \sum_{j = 0}^k {w^j G_j{}^n }$, we shall adopt the convention that $j^0=1$ for all $j$, including $j=0$.

\bigskip

For evaluating sums with $a>b$ we shall use
\[
\sum_{j=a}^b{f_j}\equiv-\sum_{j=b+1}^{a-1}{f_j}\,.
\]
In particular \[\sum_{j=a+1}^{a-1}{f_j}=-f_a\mbox{ and }\sum_{j=a}^{a-1}{f_j}=0\,.\]

\bigskip

Following Gauthier~\cite{gauthier89}, we introduce the differential operator 
\begin{equation}
D=w\frac {\rm d}{{\rm d}w}\,,
\end{equation}
and note the following properties:
\begin{enumerate}
\item $D^0u(w)=u(w)$.
\item $D(au(w)+bv(w))=aD(u(w))+bD(v(w))$, for complex numbers $a$ and $b$.
\item\label{it.kor1u1b}  $D^p \left( {u(w)v(w)} \right) = \sum_{r = 0}^p {\binom pr D^r u(w)D^{p - r} v(w)} $ for $p$~times differentiable functions $u(w)$ and $v(w)$.
\item\label{it.ffasedb} If $u(w) = \sum_{r = 0}^n {a_r w^r } $, where the coefficients $a_r$ are real or complex numbers and $n$ is an integer, then, $D^p u(w) = a_0 \delta _{p0}  + \sum_{r = 1}^n {r^p a_r w^r } $, where $\delta_{ij}$ is the usual Kronecker delta.
\item\label{it.xvq5fhg} $D^p S_k^n (w,0) = S_k^n (w,p)$\,.

\end{enumerate}
The strength of our analysis lies in the properties~\ref{it.kor1u1b}, \ref{it.ffasedb} and \ref{it.xvq5fhg} which are easily verified by induction on $p$.

\bigskip

Here is a couple of anticipated evaluations to whet the reader's appetite for reading on.
\[
\begin{split}
\sum_{j = 1}^k {j^5 G_j }&= 2671G_0  + 4322G_1  + (k^5  - 5k^4  + 50k^3  - 310k^2  + 1285k - 2671)G_k\\
&\qquad+ (k^5  - 10k^4  + 80k^3  - 500k^2  + 2080k - 4322)G_{k + 1}\,,
\end{split}
\]
\[
\sum_{j = 1}^k {( - 1)^j jG_j }  = 3G_0  - 2G_1  - ( - 1)^k (k + 3)G_k  + ( - 1)^k (k + 2)G_{k + 1}\,,
\]
\[
\sum_{j = 1}^k {\frac{{jG_j }}{{2^j }}}  = 6G_0  + 10G_1  - (k + 6)\frac{{G_k }}{{2^k }} - (k + 5)\frac{{G_{k + 1} }}{{2^{k-1} }}\,,
\]
\[
\sum_{j = 1}^k {2^j jG_j }  = \frac{2}{5}G_1  + \frac{{2^{k + 2} }}{5}kG_k  + \frac{{2^{k + 1} }}{5}(k - 1)G_{k + 1}\,,
\]
\[
\sum_{j = 1}^k {( - 1)^j G_{2j} }  = \frac{3}{5}G_0  - \frac{1}{5}G_1  + ( - 1)^k \frac{2}{5}G_{2k}  + ( - 1)^k \frac{1}{5}G_{2k + 1}\,,
\]
\[
\sum_{j = 1}^k {( - 1)^{j - 1} G_{2j - 1} }  =  - \frac{1}{5}G_0  + \frac{2}{5}G_1  + ( - 1)^k \frac{1}{5}G_{2k}  - ( - 1)^k \frac{2}{5}G_{2k + 1}\,.
\]
\[
\sum_{j = 0}^\infty {\frac{{G_j^2 }}{{2^{4j} }}}  = \frac{{3552}}{{3553}}G_0^2  + \frac{{224}}{{3553}}G_1^2  + \frac{{16}}{{3553}}G_2^2\,,
\]
\[
\begin{split}
\sum_{j = 0}^k {( - 1)^j G_{2j}^3 }  &= \frac{{23}}{{25}}G_0^3  - \frac{9}{{50}}G_1^3  - \frac{{13}}{{50}}G_2^3  + \frac{3}{{50}}G_3^3  - \frac{3}{{50}}( - 1)^k G_{2k + 1}^3\\
&\qquad+ \frac{{37}}{{50}}( - 1)^k G_{2k + 2}^3  + \frac{9}{{50}}( - 1)^k G_{2k + 3}^3  - \frac{2}{{25}}( - 1)^k G_{2k + 4}^3\,.
\end{split}
\]

\section{Evaluation of $S_k (w,r)=\sum_{j = 0}^k {w^jj^r G_j } $}
\begin{lemma}\label{lem.fxf1joy}
If $k$ is a positive integer and $w$ is a parameter, real or complex, then
\[
S_k (w,0) = \sum_{j = 0}^k {w^j G_j }  = \frac{{-(w - 1)G_0  + wG_1  - w^{k + 2} G_k  - w^{k + 1} G_{k + 1} }}{{1-w-w^2}}\,.
\]

\end{lemma}
In particular, we have
\[
S_k (1,0) = \sum_{j = 0}^k {G_j }  = G_{k + 2}  - G_1\,.
\]
\begin{proof}
Multiply through the recurrence relation $G_j=G_{j-1}+G_{j-2}$ by $w^j$ and sum over $j$, obtaining
\[
\begin{split}
\sum_{j = 0}^k {w^j G_j }  &= \sum_{j = 0}^k {w^j G_{j - 1} }  + \sum_{j = 0}^k {w^j G_{j - 2} }\\
&= w\sum_{j =  - 1}^{k - 1} {w^j G_j }  + w^2 \sum_{j =  - 2}^{k - 2} {w^j G_j } 
\end{split}
\]
\[
\begin{split}
&\quad\qquad\qquad\qquad\qquad\qquad= w\left( {\frac{{G_{ - 1} }}{w} + \sum_{j =  0}^{k} {w^j G_j }  - w^k G_k } \right)\\
&\qquad\qquad\qquad\qquad\qquad\qquad + w^2 \left( {\frac{{G_{ - 2} }}{{w^2 }} + \frac{{G_{ - 1} }}{w} + \sum_{j = 0}^k {w^j G_j }  - w^{k - 1} G_{k - 1}  - w^k G_k } \right)\,,
\end{split}
\]
so that we have
\[
\begin{split}
S_k (w,0) &= G_{ - 1}  + wS_k (w,0) - w^{k + 1} G_k\\ 
&+ G_{ - 2}  + wG_{ - 1}  + w^2 S_k (w,0)\\
&\quad - w^{k + 1} G_{k - 1}  - w^{k + 2} G_k\,,
\end{split}
\]
and the result follows.

\end{proof}

Note that in the identity of Lemma~\ref{lem.fxf1joy}, if we let $k$ approach infinity and require $w^kG_k$ to vanish as $k$ approaches infinity, then we obtain the generating function for $G_i$, namely,
\begin{equation}
S_\infty  (w,0) = \sum_{j = 0}^\infty  {w^j G_j }  = \frac{{(1-w)G_0  + wG_1}}{{1 - w - w^2 }}\,.
\end{equation}
\begin{lemma}\label{lem.piatqhb}
The rational function 
\[
A(w;m) = D^m \frac{1}{{1-w-w^2}}=D^mA(w;0)\,,
\]
satisfies the following recursion relation
\[
({1-w-w^2})A(w;m) =\delta_{m0}  + w\sum_{j = 0}^{m - 1} {\binom mj(2^{m - j} w + 1)A(w;j)},\quad m=0,1,2,\ldots\,.
\]

\end{lemma}
\begin{proof}
Write $(1-w-w^2)A(w;0)=1$ and apply $D^m$ to both sides.
\end{proof}
\begin{thm}\label{thm.lou2j9u}
If $k$ and $r$ are nonnegative integers and $w$ a real or complex parameter, then
\[
\begin{split}
S_k (w,r) = \sum_{j = 0}^k {w^j j^r G_j }&=  - G_0 \sum_{m = 0}^r {\binom rm(w-\delta _{rm})A(w;m)}  + wG_1 \sum_{m = 0}^r {\binom rmA(w;m)}\\ 
&\quad- w^{k + 2} G_k \sum_{m = 0}^r {\binom rm(k + 2)^m A(w;r - m)}\\ 
&\quad- w^{k + 1} G_{k + 1} \sum_{m = 0}^r {\binom rm(k + 1)^m A(w;r - m)}\,,
\end{split}
\]
where the $A(w;p)$, $p=0,1,2,\ldots,r$, are given recursively by
\[
({1-w-w^2})A(w;p) =\delta_{p0}  + w\sum_{j = 0}^{p - 1} {\binom pj(2^{p - j} w + 1)A(w;j)}\,.
\]

\end{thm}
\begin{proof}
Apply $D^r$ to both sides of the identity of Lemma~\ref{lem.fxf1joy} and make use of Lemma~\ref{lem.piatqhb} and the properties~\ref{it.kor1u1b}, \ref{it.ffasedb} and \ref{it.xvq5fhg} of the $D$ operator.
\end{proof}
Numerous interesting summation formulas can be obtained directly from Theorem~\ref{thm.lou2j9u}. By setting $w=1/2$, $r=0,1,2$, respectively, we already have some examples:
\begin{equation}
\sum_{j = 0}^k {\frac{{G_j }}{{2^j }}}  = 2G_0  + 2G_1  - \frac{{G_k }}{{2^k }} - \frac{{G_{k + 1} }}{{2^{k - 1} }}\,,
\end{equation}
corresponding to Vajda~\cite{vajda}, Formula~(37), proved by induction, which in the limit as $k\to\infty$ gives
\begin{equation}
\sum_{j = 0}^\infty  {\frac{{G_j }}{{2^j }}}  = 2G_0  + 2G_1\,.
\end{equation}
\begin{equation}
\sum_{j = 0}^k {\frac{{jG_j }}{{2^j }}}  = 6G_0  + 10G_1  - (k + 6)\frac{{G_k }}{{2^k }} - (k + 5)\frac{{G_{k + 1} }}{{2^{k-1} }}\,,
\end{equation}
\begin{equation}\label{eq.s4zoecu}
\sum_{j = 0}^\infty {\frac{{jG_j }}{{2^j }}}  = 6G_0  + 10G_1\,.
\end{equation}
\begin{equation}
\begin{split}
\sum_{j = 0}^k {\frac{{j^2G_j }}{{2^j }} }  &= 58G_0  + 94G_1  - (k^2  + 12k + 58)\frac{{G_k }}{{2^k }}\\
&\qquad - (k^2  + 10k + 47)\frac{{G_{k + 1} }}{{2^{k - 1} }}\,,
\end{split}
\end{equation}
\begin{equation}
\sum_{j = 0}^\infty {\frac{{j^2G_j }}{{2^j }}}  = 58G_0  + 94G_1\,.
\end{equation}
Note that the identity~\eqref{eq.s4zoecu} subsumes Formulas~(A3.52) and (A3.53) of Dunlap~\cite{dunlap}.

\bigskip

If $i=\sqrt{-1}$ is the imaginary unit, then using $f_j=i^j j^r G_j$ in the summation identity
\[
\sum_{j = 0}^{2k} {f_j }  = \sum_{j = 0}^k {f_{2j} }  + \sum_{j = 1}^k {f_{2j - 1} } 
\]
allows us to write
\begin{equation}\label{eq.f1d4elc}
\sum_{j = 0}^{2k} {i^j j^r G_j }  = 2^r \sum_{j = 0}^k {( - 1)^j j^r G_{2j} }  + i\sum_{j = 1}^k {( - 1)^{j - 1} (2j - 1)^r G_{2j - 1} }\,.
\end{equation}
Thus, setting $w=i$ in Theorem~\ref{thm.lou2j9u} and making use of identity~\eqref{eq.f1d4elc}, we have the following results for $r=0,1$.
\begin{equation}\label{eq.s8gkjr0}
\sum_{j = 0}^k {( - 1)^j G_{2j} }  = \frac{3}{5}G_0  - \frac{1}{5}G_1  + ( - 1)^k \frac{2}{5}G_{2k}  + ( - 1)^k \frac{1}{5}G_{2k + 1}\,,
\end{equation}
\begin{equation}\label{eq.qm5qr44}
\sum_{j = 1}^k {( - 1)^{j - 1} G_{2j - 1} }  =  - \frac{1}{5}G_0  + \frac{2}{5}G_1  + ( - 1)^k \frac{1}{5}G_{2k}  - ( - 1)^k \frac{2}{5}G_{2k + 1}\,,
\end{equation}
\begin{equation}
\sum_{j = 0}^k {( - 1)^j jG_{2j} }  =  - \frac{1}{5}G_0  + \frac{1}{5}( - 1)^k (2k + 1)G_{2k}  + \frac{1}{5}( - 1)^k kG_{2k + 1}\,,
\end{equation}
\begin{equation}
\begin{split}
\sum_{j = 1}^k {( - 1)^{j - 1} jG_{2j - 1} }  &=  - \frac{1}{5}G_0  + \frac{1}{5}G_1  + \frac{1}{5}( - 1)^k (k + 1)G_{2k}\\
&\quad - \frac{1}{5}( - 1)^k (2k + 1)G_{2k + 1}\,.
\end{split}
\end{equation}
Identities~\eqref{eq.s8gkjr0} and \eqref{eq.qm5qr44} are the alternating counterparts of Formulas~(34) and (35) of Vajda~\cite{vajda}.
\begin{cor}\label{cor.tf87enm}
If $k$ is a nonnegative integer and $w$ is a real or complex parameter, then
\[
\begin{split}
\sum_{j = 0}^k {w^j jG_j }  &= \frac{{2 - w}}{{(1 - w - w^2)^2 }}w^2 G_0  + \frac{{w^2  + 1}}{{(1 - w - w^2)^2 }}wG_1\\ 
&\quad + \frac{{kw^2  + (k + 1)w - (k + 2)}}{{(1 - w - w^2)^2 }}w^{k + 2} G_k\\ 
&\quad + \frac{{(k - 1)w^2  + kw - (k + 1)}}{{(1 - w - w^2)^2 }}w^{k + 1} G_{k + 1}\,. 
\end{split}
\]
\end{cor}
In particular, we have
\begin{equation}
\sum_{j = 0}^k {2^j jG_j }  = \frac{2}{5}G_1  + \frac{{2^{k + 2} }}{5}kG_k  + \frac{{2^{k + 1} }}{5}(k - 1)G_{k + 1}
\end{equation}
and
\begin{equation}
\begin{split}
\sum_{j = 0}^k {3^j jG_j }  &=  - \frac{9}{{121}}G_0  + \frac{{30}}{{121}}G_1  + \frac{{3^{k + 2} }}{{121}}(11k + 1)G_k\\
&\qquad+ \frac{{3^{k + 1} }}{{121}}(11k - 10)G_{k + 1}\,.
\end{split}
\end{equation}
\begin{cor}\label{cor.ap9or90}
If $r$ is a nonnegative integer and $w$ is a real or complex parameter such that $w^kG_k\to 0$ as $k\to\infty$ for all positive integers $k$, then
\[
\begin{split}
S_\infty (w,r) = \sum_{j = 0}^\infty {w^j j^r G_j }&=  - G_0 \sum_{m = 0}^r {\binom rm(w-\delta _{rm})A(w;m)}\\
&\qquad + wG_1 \sum_{m = 0}^r {\binom rmA(w;m)}\,.
\end{split}
\]

\end{cor}
Thus, Corollary~\ref{cor.ap9or90} provides the evaluation of generating functions for $j^rG_j$. A strategy for using probability to evaluate $S_\infty (1/2,r)$ was discussed in reference~\cite{benjamin}. Dropping the last two terms in Corollary~\ref{cor.tf87enm}, we have
\begin{equation}
S_\infty (w,1)=\sum_{j = 0}^\infty {w^j jG_j }  = \frac{{2 - w}}{{(w^2  + w - 1)^2 }}w^2 G_0  + \frac{{w^2  + 1}}{{(w^2  + w - 1)^2 }}wG_1\,,
\end{equation}
a result that was also obtained and whose convergence was exhaustively discussed by Glaister~\cite{glaister95,glaister96} for the Fibonacci case, (G=F).
\begin{cor}\label{cor.hba6sh9}
If $k$ and $r$ are nonnegative integers, then
\[
\begin{split}
S_k (r) = \sum_{j = 0}^k {j^r G_j }&=  - G_0 \sum_{m = 0}^r {\binom rm(1-\delta _{rm})A(m)}  + G_1 \sum_{m = 0}^r {\binom rmA(m)}\\ 
&\quad- G_k \sum_{m = 0}^r {\binom rm(k + 2)^m A(r - m)}\\ 
&\quad- G_{k + 1} \sum_{m = 0}^r {\binom rm(k + 1)^m A(r - m)}\,,
\end{split}
\]
where the $A(m)$ satisfy the recurrence relation
\[
A(m) =-\delta_{m0}  - \sum_{j = 0}^{m - 1} {\binom mj(2^{m - j}  + 1)A(j)} ,\quad m=0,1,2,\ldots
\]

\end{cor}
Here is a little list from Corollary~\ref{cor.hba6sh9}.
\begin{equation}
\sum_{j = 0}^k {G_j }  =  - G_1  + G_k  + G_{k + 1}\,,
\end{equation}
\begin{equation}
\sum_{j = 0}^k {jG_j }  = G_0  + 2G_1  + (k - 1)G_k  + (k - 2)G_{k + 1}\,,
\end{equation}
\begin{equation}
\sum_{j = 0}^k {j^2 G_j }  =  - 5G_0  - 8G_1  + (k^2 - 2k + 5)G_k  + (k^2 - 4k + 8 )G_{k + 1}\,,
\end{equation}
\begin{equation}
\begin{split}
\sum_{j = 0}^k {j^3 G_j }  &= 31G_0  + 50G_1  + (k^3  - 3k^2  + 15k - 31)G_k\\
&\qquad+ (k^3  - 6k^2  + 24k - 50)G_{k + 1}\,,
\end{split}
\end{equation}
\begin{equation}
\begin{split}
\sum_{j = 0}^k {j^4 G_j }  &=  - 257G_0  - 416G_1  + (k^4  - 4k^3  + 30k^2  - 124k + 257)G_k\\
&\qquad+ (k^4  - 8k^3  + 48k^2  - 200k + 416)G_{k + 1}\,,
\end{split}
\end{equation}
\begin{equation}
\begin{split}
\sum_{j = 0}^k {j^5 G_j }&= 2671G_0  + 4322G_1  + (k^5  - 5k^4  + 50k^3  - 310k^2  + 1285k - 2671)G_k\\
&\qquad+ (k^5  - 10k^4  + 80k^3  - 500k^2  + 2080k - 4322)G_{k + 1}\,,
\end{split}
\end{equation}
\begin{equation}
\begin{split}
\sum_{j = 0}^k {j^6 G_j }  &=  - 33305G_0  - 53888G_1\\ 
&\qquad+ (k^6  - 6k^5  + 75k^4  - 620k^3  + 3855k^2  - 16026k + 33305)G_k\\
&\qquad + (k^6  - 12k^5  + 120k^4  - 1000k^3  + 6240k^2  - 25932k + 53888)G_{k + 1}\,.
\end{split}
\end{equation}
\begin{cor}\label{cor.zym4qge}
If $k$ and $r$ are nonnegative integers, then
\[
\begin{split}
\overline S_k(r)=S_k (-1,r) &= \sum_{j = 0}^k {(-1)^j j^r G_j }\\
&=   G_0 \sum_{m = 0}^r {\binom rm(\delta _{rm}  + 1)\overline A(m)}  - G_1 \sum_{m = 0}^r {\binom rm\overline A(m)}\\ 
&\quad- (-1)^{k} G_k \sum_{m = 0}^r {\binom rm(k + 2)^m \overline A(r - m)}\\ 
&\quad+ (-1)^{k} G_{k + 1} \sum_{m = 0}^r {\binom rm(k + 1)^m \overline A(r - m)}\,,
\end{split}
\]
where the $\overline A(p)$ are given recursively by
\[
\overline A(p) =  \delta_{p0}+\sum_{j = 0}^{p - 1} {\binom pj(2^{p - j}  - 1)\overline A(j)} ,\quad p=1,2,\ldots
\]

\end{cor}
Here are a few evaluations.
\begin{equation}
\sum_{j = 0}^k {( - 1)^j G_j }  = 2G_0  - G_1  - ( - 1)^k G_k  + ( - 1)^k G_{k + 1}\,,
\end{equation}
\begin{equation}
\sum_{j = 0}^k {( - 1)^j jG_j }  = 3G_0  - 2G_1  - ( - 1)^k (k + 3)G_k  + ( - 1)^k (k + 2)G_{k + 1}\,,
\end{equation}
\begin{equation}
\begin{split}
\sum_{j = 0}^k {( - 1)^j j^2G_j }  &= 13G_0  - 8G_1  - ( - 1)^k (k^2  + 6k + 13)G_k\\
&\qquad+ ( - 1)^k (k^2  + 4k + 8)G_{k + 1}\,,
\end{split}
\end{equation}
\begin{equation}
\begin{split}
\sum_{j = 0}^k {( - 1)^j j^3 G_j }  &= 81G_0  - 50G_1\\ 
&\quad - ( - 1)^k (k^3  + 9k^2  + 39k + 81)G_k\\ 
&\qquad + ( - 1)^k (k^3  + 6k^2  + 24k + 50)G_{k + 1}\,.
\end{split}
\end{equation}

\section{Evaluation of $S_k^n (w,r)=\sum_{j = 0}^k {w^jj^r G_j{}^n } $}
\begin{lemma}\label{lem.tspbkzv}
If $k$ and $n$ are nonnegative integers and $w$ is a parameter, real or complex, then
\[
\begin{split}
S_k^n (w,0)=\sum_{j = 0}^k {w^j G_j^n }  &= \left( {\sum_{s = 0}^{n + 1} {\left\{{\binom{n+1}s}_F( - 1)^{\left\lceil {(n - s + 1)/2} \right\rceil } \sum_{j = 0}^{s - 1} {w^{j + n - s + 1} G_j^n } \right\}} } \right.\\
&\qquad\qquad{\left. - \sum_{s = 0}^{n + 1} {\left\{{\binom{n+1}s}_F( - 1)^{\left\lceil {(n - s + 1)/2} \right\rceil } \sum_{j = k + 1}^{k + s} {w^{j + n - s + 1} G_j^n }\right\} }\right) }\\
&\qquad\qquad\qquad{{} \mathord{\left/
 {\vphantom {{} {\sum_{s = 0}^{n + 1} {{\binom{n+1}s}_F( - 1)^{\left\lceil {(n - s + 1)/2} \right\rceil } w^{n - s + 1} } }}} \right.
 \kern-\nulldelimiterspace} {\sum_{s = 0}^{n + 1} {{\binom{n+1}s}_F( - 1)^{\left\lceil {(n - s + 1)/2} \right\rceil } w^{n - s + 1} } }}\,,
 \end{split}
\]
where $\lceil u\rceil$ is the smallest integer greater than $u$ and the Fibonomial coefficients are defined by
\[
{\binom pq}_F = \frac{{F_p F_{p - 1}  \cdots F_{p - q + 1} }}{{F_q F_{q - 1}  \cdots F_1 }} = \prod_{j = 1}^q {\frac{{F_{p - q + j} }}{{F_j }}}\,. 
\]

\end{lemma}
\begin{proof}
Multiply through the following identity~\cite[\mbox{section 1.2.8, Exercise~30}]{knuth97} by $w^j$ 
\[
\sum_{s = 0}^{n + 1} {\binom{n+1}s_F( - 1)^{\left\lceil {(n - s + 1)/2} \right\rceil } G_{j + s}^n }  = 0
\]
and sum over $j$, obtaining
\[
\sum_{s = 0}^{n + 1} {{\binom{n+1}s}_F( - 1)^{\left\lceil {(n - s + 1)/2} \right\rceil } \sum_{j = 0}^k {w^j G_{j + s}^n } }  = 0\,,
\]
which after shifting the index $j$ becomes
\[
\sum_{s = 0}^{n + 1} {\binom{n+1}s_F( - 1)^{\left\lceil {(n - s + 1)/2} \right\rceil } w^{-s}\sum_{j = s}^{k + s} {w^j G_j^n } }  = 0\,.
\]
Since
\[
\sum_{j = s}^{k + s} {w^j G_j^n }  = \sum_{j = 0}^k {w^j G_{j}^n }  + \sum_{j = k + 1}^{k + s} {w^j G_{j}^n }  - \sum_{j = 0}^{s - 1} {w^j G_{j}^n }\,,
\]
we have
\[
\sum_{s = 0}^{n + 1} {\binom{n+1}s_F( - 1)^{\left\lceil {(n - s + 1)/2} \right\rceil } w^{ - s} \left( {\sum_{j = 0}^k {w^j G_{j}^n }  + \sum_{j = k + 1}^{k + s} {w^j G_{j}^n }  - \sum_{j = 0}^{s - 1} {w^j G_{j}^n } } \right)}  = 0\,.
\]
Thus,
\[
\sum_{s = 0}^{n + 1} {\binom{n+1}s_F( - 1)^{\left\lceil {(n - s + 1)/2} \right\rceil } w^{ - s} \left( {S_k^n(w,0)  + \sum_{j = k + 1}^{k + s} {w^j G_{j}^n }  - \sum_{j = 0}^{s - 1} {w^j G_{j}^n } } \right)}  = 0\,.
\]
Clearing brackets, multiplying through by $w^{n+1}$ and making $S_k^n(w,0)$ the subject, the lemma is proved.
\end{proof}
In particular, we have
\begin{equation}\label{eq.araqz6y}
\begin{split}
\sum_{j = 0}^k {w^j G_j^2 }  &=  - \frac{{(2w^2  + 2w - 1)G_0^2 }}{{w^3  - 2w^2  - 2w + 1}} - \frac{{(2w^2  - w)G_1^2 }}{{w^3  - 2w^2  - 2w + 1}}\\
&\qquad + \frac{{w^2 G_2^2 }}{{w^3  - 2w^2  - 2w + 1}} + \frac{{(2w^2  + 2w - 1)w^{k + 1} G_{k + 1}^2 }}{{w^3  - 2w^2  - 2w + 1}}\\
&\qquad + \frac{{(2w - 1)w^{k + 2} G_{k + 2}^2 }}{{w^3  - 2w^2  - 2w + 1}} - \frac{{w^{k + 3} G_{k + 3}^2 }}{{w^3  - 2w^2  - 2w + 1}}
\end{split}
\end{equation}
and
\begin{equation}\label{eq.kb3bild}
\begin{split}
\sum_{j = 0}^k {w^j G_j^3 }  &= \frac{{(3w^3  - 6w^2  - 3w + 1)G_0^3 }}{{w^4  + 3w^3  - 6w^2  - 3w + 1}} - \frac{{(6w^3  + 3w^2  - w)G_1^3 }}{{w^4  + 3w^3  - 6w^2  - 3w + 1}}\\
&\qquad - \frac{{(3w^3  - w^2 )G_2^3 }}{{w^4  + 3w^3  - 6w^2  - 3w + 1}} + \frac{{w^3 G_3^3 }}{{w^4  + 3w^3  - 6w^2  - 3w + 1}}\\
&\qquad - \frac{{(3w^3  - 6w^2  - 3w + 1)w^{k + 1} G_{k + 1}^3 }}{{w^4  + 3w^3  - 6w^2  - 3w + 1}} + \frac{{(6w^2  + 3w - 1)w^{k + 2} G_{k + 2}^3 }}{{w^4  + 3w^3  - 6w^2  - 3w + 1}}\\
&\qquad + \frac{{(3w - 1)w^{k + 3} G_{k + 3}^3 }}{{w^4  + 3w^3  - 6w^2  - 3w + 1}} - \frac{{w^{k + 4} G_{k + 4}^3 }}{{w^4  + 3w^3  - 6w^2  - 3w + 1}}\,.
\end{split}
\end{equation}
Interesting evaluations can already be obtained from~\eqref{eq.araqz6y} and~\eqref{eq.kb3bild}. For example, at $w=1$, we have
\begin{equation}
\sum_{j = 0}^k {G_j^2 }  = \frac{3}{2}G_0^2  + \frac{1}{2}G_1^2  - \frac{1}{2}G_2^2  - \frac{3}{2}G_{k + 1}^2  - \frac{1}{2}G_{k + 2}^2  + \frac{1}{2}G_{k + 3}^2
\end{equation}
and
\begin{equation}
\begin{split}
\sum_{j = 0}^k {G_j^3 }  &= \frac{5}{4}G_0^3  + 2G_1^3  + \frac{1}{2}G_2^3  - \frac{1}{4}G_3^3\\
&\qquad- \frac{5}{4}G_{k + 1}^3  - 2G_{k + 2}^3  - \frac{1}{2}G_{k + 3}^3  + \frac{1}{4}G_{k + 4}^3\,,
\end{split}
\end{equation}
while $w=1/2$ gives
\begin{equation}
\sum_{j = 0}^k {\frac{{G_j^2 }}{{2^j }}}  = \frac{4}{3}G_0^2  - \frac{2}{3}G_2^2  - \frac{1}{3}\frac{{G_{k + 1}^2 }}{{2^{k - 1} }} + \frac{1}{3}\frac{{G_{k + 3}^2 }}{{2^k }}
\end{equation}
and
\begin{equation}
\begin{split}
\sum_{j = 0}^k {\frac{{G_j^3 }}{{2^j }}}  &= \frac{{26}}{{25}}G_0^3  + \frac{{16}}{{25}}G_1^3  + \frac{2}{{25}}G_2^3  - \frac{2}{{25}}G_3^3\\
&\quad - \frac{{13}}{{25}}\frac{{G_{k + 1}^3 }}{{2^k }} - \frac{8}{{25}}\frac{{G_{k + 2}^3 }}{{2^k }} - \frac{1}{{25}}\frac{{G_{k + 3}^3 }}{{2^k }} + \frac{1}{{25}}\frac{{G_{k + 4}^3 }}{{2^k }}\,.
\end{split}
\end{equation}
At $w=2$, we have
\begin{equation}
\begin{split}
\sum_{j = 0}^k {2^j G_j^2 }  &= \frac{{11}}{3}G_0^2  + 2G_1^2  - \frac{4}{3}G_2^2  - \frac{{11}}{3}2^{k + 1} G_{k + 1}^2\\
&\quad- 2^{k + 2} G_{k + 2}^2  + \frac{1}{3}2^{k + 3} G_{k + 3}^2
\end{split}
\end{equation}
and
\begin{equation}
\begin{split}
\sum_{j = 0}^k {2^j G_j^3 }  &=  - \frac{5}{{11}}G_0^3  - \frac{{58}}{{11}}G_1^3  - \frac{{20}}{{11}}G_2^3  + \frac{8}{{11}}G_3^3  + \frac{5}{{11}}2^{k + 1} G_{k + 1}^3\\
&\qquad+ \frac{{29}}{{11}}2^{k + 2} G_{k + 2}^3  + \frac{5}{{11}}2^{k + 3} G_{k + 3}^3  - \frac{1}{{11}}2^{k + 4} G_{k + 4}^3\,.
\end{split}
\end{equation}
At $w=1/16$, we have
\begin{equation}
\begin{split}
\sum_{j = 0}^k {\frac{{G_j^2 }}{{2^{4j} }}}  &= \frac{{3552}}{{3553}}G_0^2  + \frac{{224}}{{3553}}G_1^2  + \frac{{16}}{{3553}}G_2^2\\
&\quad- \frac{{222}}{{3553}}\frac{{G_{k + 1}^2 }}{{2^{4k} }} - \frac{7}{{3553}}\frac{{G_{k + 2}^2 }}{{2^{4k - 1} }} - \frac{1}{{3553}}\frac{{G_{k + 3}^2 }}{{2^{4k} }}\,,
\end{split}
\end{equation}
giving
\begin{equation}
\sum_{j = 0}^\infty {\frac{{G_j^2 }}{{2^{4j} }}}  = \frac{{3552}}{{3553}}G_0^2  + \frac{{224}}{{3553}}G_1^2  + \frac{{16}}{{3553}}G_2^2\,.
\end{equation}
Evaluation at $w=\sqrt {-1}=i$, $n=2,3$ gives
\begin{equation}
\begin{split}
\sum_{j = 0}^k {( - 1)^j G_{2j}^2 } &= \frac{5}{6}G_0^2  + \frac{1}{6}G_1^2  - \frac{1}{6}G_2^2  + \frac{{( - 1)^k }}{6}G_{2k + 1}^2\\
&\qquad + \frac{{( - 1)^k }}{2}G_{2k + 2}^2  - \frac{{( - 1)^k }}{6}G_{2k + 3}^2\,,
\end{split}
\end{equation}
\begin{equation}
\begin{split}
\sum_{j = 1}^k {( - 1)^{j - 1} G_{2j - 1}^2 }  &= \frac{1}{6}G_0^2  + \frac{1}{2}G_1^2  - \frac{1}{6}G_2^2  - \frac{5}{6}( - 1)^k G_{2k + 1}^2\\
&\qquad- \frac{{( - 1)^k }}{6}G_{2k + 2}^2  + \frac{{( - 1)^k }}{6}G_{2k + 3}^2\,,
\end{split}
\end{equation}
\begin{equation}
\begin{split}
\sum_{j = 0}^k {( - 1)^j G_{2j}^3 }  &= \frac{{23}}{{25}}G_0^3  - \frac{9}{{50}}G_1^3  - \frac{{13}}{{50}}G_2^3  + \frac{3}{{50}}G_3^3  - \frac{3}{{50}}( - 1)^k G_{2k + 1}^3\\
&\qquad+ \frac{{37}}{{50}}( - 1)^k G_{2k + 2}^3  + \frac{9}{{50}}( - 1)^k G_{2k + 3}^3  - \frac{2}{{25}}( - 1)^k G_{2k + 4}^3\,,
\end{split}
\end{equation}
and
\begin{equation}
\begin{split}
\sum_{j = 1}^k {( - 1)^{j - 1} G_{2j - 1}^3 }  &=  - \frac{3}{{50}}G_0^3  + \frac{{37}}{{50}}G_1^3  + \frac{9}{{50}}G_2^3  - \frac{2}{{25}}G_3^3  - \frac{{23}}{{25}}( - 1)^k G_{2k + 1}^3\\
&\qquad+ \frac{9}{{50}}( - 1)^k G_{2k + 2}^3  + \frac{{13}}{{50}}( - 1)^k G_{2k + 3}^3  - \frac{3}{{50}}( - 1)^k G_{2k + 4}^3\,.
\end{split}
\end{equation}
\begin{cor}[Generating function for $G_j{}^n$]\label{cor.w8th5gl}
If $n$ is a nonnegative integer and $w$ is a parameter, real or complex, such that $w^kG_k{}^n$ vanishes as $k$ approaches infinity, then
\[
\begin{split}
S_\infty^n (w,0)=\sum_{j = 0}^\infty {w^j G_j^n }  &= {\sum_{s = 0}^{n + 1} {\left\{{\binom{n+1}s}_F( - 1)^{\left\lceil {(n - s + 1)/2} \right\rceil } \sum_{j = 0}^{s - 1} {w^{j + n - s + 1} G_j^n } \right\}} }\\
&\qquad\qquad\qquad{{} \mathord{\left/
 {\vphantom {{} {\sum_{s = 0}^{n + 1} {{\binom{n+1}s}_F( - 1)^{\left\lceil {(n - s + 1)/2} \right\rceil } w^{n - s + 1} } }}} \right.
 \kern-\nulldelimiterspace} {\sum_{s = 0}^{n + 1} {{\binom{n+1}s}_F( - 1)^{\left\lceil {(n - s + 1)/2} \right\rceil } w^{n - s + 1} } }}\,.
 \end{split}
\]

\end{cor}
In particular, the generating functions of $G_j{}^2$ and $G_j{}^3$ are, respectively,
\begin{equation}
S_\infty^2 (w,0)=\sum_{j = 0}^\infty {w^j G_j^2 }  = \frac{{w^2 G_2^2  - (2w^2  + 2w - 1)G_0^2  - (2w^2  - w)G_1^2 }}{{w^3  - 2w^2  - 2w + 1}}
\end{equation}
and
\begin{equation}
\begin{split}
S_\infty^3 (w,0)=\sum_{j = 0}^\infty {w^j G_j^3 }  &= \frac{{(3w^3  - 6w^2  - 3w + 1)G_0^3 }}{{w^4  + 3w^3  - 6w^2  - 3w + 1}} - \frac{{(6w^3  + 3w^2  - w)G_1^3 }}{{w^4  + 3w^3  - 6w^2  - 3w + 1}}\\
&\qquad - \frac{{(3w^3  - w^2 )G_2^3 }}{{w^4  + 3w^3  - 6w^2  - 3w + 1}} + \frac{{w^3 G_3^3 }}{{w^4  + 3w^3  - 6w^2  - 3w + 1}}\,.
\end{split}
\end{equation}
\begin{lemma}\label{lem.q25ynvb}
The rational function  
\[
A_n (w;m) = D^m A_n (w;0) = D^m \;\frac{1}{{\sum_{s = 0}^{n + 1} {\binom {n+1}s_F( - 1)^{\left\lceil {(n - s + 1)/2} \right\rceil } w^{n - s + 1} } }}
\]

satisfies the following recurrence relation
\[
\begin{split}
&A_n (w;m)\sum_{s = 0}^{n + 1} {\binom {n+1}s_F( - 1)^{\left\lceil {(n - s + 1)/2} \right\rceil } w^{n - s + 1} }\\
&\qquad= \delta _{m0}  - \sum_{j = 0}^{m - 1} {\binom mjA_n (w;j)\sum_{s = 0}^n {\binom {n+1}s_F( - 1)^{\left\lceil {(n - s + 1)/2} \right\rceil } (n - s + 1)^{m - j} w^{n - s + 1} } }\,.
\end{split}
\]

\end{lemma}
\begin{proof}
Write
\[
A_n (w;0)\left( {1 + \sum_{s = 0}^n {\binom {n+1}s_F( - 1)^{\left\lceil {(n - s + 1)/2} \right\rceil } w^{n - s + 1} } } \right) = 1\,.
\]
Apply $D^m$ to both sides to obtain
\[
\sum_{j = 0}^m {\binom mjD^j A_n (w;0)\left( {\delta _{jm}  + \sum_{s = 0}^n {\binom {n+1}s_F( - 1)^{\left\lceil {(n - s + 1)/2} \right\rceil } (n - s + 1)^{m - j} w^{n - s + 1} } } \right)}  = \delta _{m0}\,,
\]
from which the result follows.

\end{proof}

\begin{thm}\label{thm.main}
If $k$, $r$ and $n$ are nonnegative integers and $w$ is a real or complex parameter, then
\[
\begin{split}
S_k{}^n (w,r) &= \sum_{j = 0}^k {w^j j^r G_j{}^n }\\  
&= A_n (w;r)G_0{}^n  + \sum_{m = 0}^r \binom rmA_n (w;m){\sum_{j = 1}^n {j^{r - m} w^jG_j{}^n } }\\
&\quad + \sum_{m = 0}^r {\binom rmA_n (w;m)\sum_{s = 0}^n {\binom {n+1}s_F( - 1)^{\left\lceil {(n - s + 1)/2} \right\rceil } \sum_{j = 0}^{s - 1} {(j + n - s + 1)^{r - m} w^{j + n - s + 1} G_j{}^n } } }\\ 
&\quad + \sum_{m = 0}^r {\binom rmA_n (w;m)\sum_{s = 0}^{n + 1} {\binom {n+1}s_F( - 1)^{\left\lceil {(n - s + 1)/2} \right\rceil } \sum_{j = k + 1}^{k + s} {(j + n - s + 1)^{r - m} w^{j + n - s + 1} G_j{}^n } } }\,,
\end{split}
\]
where the functions $A_n(w,m)$ are given recursively by
\[
\begin{split}
&A_n (w;m)\sum_{s = 0}^{n + 1} {\binom {n+1}s_F( - 1)^{\left\lceil {(n - s + 1)/2} \right\rceil } w^{n - s + 1} }\\
&\qquad= \delta _{m0}  - \sum_{j = 0}^{m - 1} {\binom mjA_n (w;j)\sum_{s = 0}^n {\binom {n+1}s_F( - 1)^{\left\lceil {(n - s + 1)/2} \right\rceil } (n - s + 1)^{m - j} w^{n - s + 1} } }\,,
\end{split}
\]
for $m=0,1,2,\ldots,r$.
\end{thm}
\begin{proof}
Apply $D^r$ to both sides of the identity of Lemma~\ref{lem.tspbkzv} and make use of Lemma~\ref{lem.q25ynvb} and the properties~\ref{it.kor1u1b}, \ref{it.ffasedb} and \ref{it.xvq5fhg} of the $D$ operator.
\end{proof}

\end{document}